\documentclass[12pt]{article}
\usepackage{graphicx}
\usepackage{amsmath}
\usepackage{amsfonts}
\usepackage{amsthm}
\usepackage[T1]{fontenc}
\usepackage{url}
\usepackage{color}
\usepackage[margin=1in]{geometry}
\usepackage{amssymb,bm}
\usepackage{amsmath}
\usepackage{amsthm}
\usepackage{graphicx}
\usepackage[active]{srcltx} 
\usepackage{hyperref}
\hypersetup{pdfborder=0 0 0}

\newtheorem{theorem}{{\sc Theorem}}[section]

\newtheorem{remark}[theorem]{Remark}

\newtheorem{definition}[theorem]{Definition}

\newcommand{\bb}[1]{\mathbb{ #1}}

\newcommand{\dif}[2]{\displaystyle\frac{\partial #1}{\partial #2}}
\newcommand{\Grad}{\nabla}

\newcommand{\av}[1]{\langle #1 \rangle}

\def\XXint#1#2#3{{\setbox0=\hbox{$#1{#2#3}{\int}$ }
\vcenter{\hbox{$#2#3$ }}\kern-.6\wd0}}

\newcommand{\Gk}{\kappa}

\newcommand{\Gl}{\lambda}

\newcommand{\Gth}{\theta}

\newcommand{\GO}{\Omega}

\bmdefine\BGa{\alpha}
\bmdefine\BGb{\beta}
\bmdefine\BGd{\delta}
\bmdefine\BGe{\epsilon}
\bmdefine\BGve{\varepsilon}
\bmdefine\BGf{\phi}
\bmdefine\BGvf{\varphi}
\bmdefine\BGg{\gamma}
\bmdefine\BGc{\chi}
\bmdefine\BGi{\iota}
\bmdefine\BGk{\kappa}
\bmdefine\BGl{\lambda}
\bmdefine\BGn{\eta}
\bmdefine\BGm{\mu}
\bmdefine\BGv{\nu}
\bmdefine\BGp{\pi}
\bmdefine\BGth{\theta}
\bmdefine\BGvth{\vartheta}
\bmdefine\BGr{\rho}
\bmdefine\BGvr{\varrho}
\bmdefine\BGs{\sigma}
\bmdefine\BGvs{\varsigma}
\bmdefine\BGt{\tau}
\bmdefine\BGj{\tau}
\bmdefine\BGu{\upsilon}
\bmdefine\BGo{\omega}
\bmdefine\BGx{\xi}
\bmdefine\BGy{\psi}
\bmdefine\BGz{\zeta}
\bmdefine\BGD{\Delta}
\bmdefine\BGF{\Phi}
\bmdefine\BGG{\Gamma}
\bmdefine\BGL{\Lambda}
\bmdefine\BGP{\Pi}
\bmdefine\BGT{\Theta}
\bmdefine\BGS{\Sigma}
\bmdefine\BGU{\Upsilon}
\bmdefine\BGO{\Omega}
\bmdefine\BGX{\Xi}
\bmdefine\BGY{\Psi}

\newcommand{\CA}{{\mathcal A}}

\bmdefine\BCA{{\mathcal A}}
\bmdefine\BCB{{\mathcal B}}
\bmdefine\BCC{{\mathcal C}}
\bmdefine\BCD{{\mathcal D}}
\bmdefine\BCE{{\mathcal E}}
\bmdefine\BCF{{\mathcal F}}
\bmdefine\BCG{{\mathcal G}}
\bmdefine\BCH{{\mathcal H}}
\bmdefine\BCI{{\mathcal I}}
\bmdefine\BCJ{{\mathcal J}}
\bmdefine\BCK{{\mathcal K}}
\bmdefine\BCL{{\mathcal L}}
\bmdefine\BCM{{\mathcal M}}
\bmdefine\BCN{{\mathcal N}}
\bmdefine\BCO{{\mathcal O}}
\bmdefine\BCP{{\mathcal P}}
\bmdefine\BCQ{{\mathcal Q}}
\bmdefine\BCR{{\mathcal R}}
\bmdefine\BCS{{\mathcal S}}
\bmdefine\BCT{{\mathcal T}}
\bmdefine\BCU{{\mathcal U}}
\bmdefine\BCV{{\mathcal V}}
\bmdefine\BCW{{\mathcal W}}
\bmdefine\BCX{{\mathcal X}}
\bmdefine\BCY{{\mathcal Y}}
\bmdefine\BCZ{{\mathcal Z}}

\bmdefine\Bzr{ 0}
\bmdefine\Ba{ a}
\bmdefine\Bb{ b}
\bmdefine\Bc{ c}
\bmdefine\Bd{ d}
\bmdefine\Be{ e}
\bmdefine\Bf{ f}
\bmdefine\Bg{ g}
\bmdefine\Bh{ h}
\bmdefine\Bi{ i}
\bmdefine\Bj{ j}
\bmdefine\Bk{ k}
\bmdefine\Bl{ l}
\bmdefine\Bm{ m}
\bmdefine\Bn{ n}
\bmdefine\Bo{ o}
\bmdefine\Bp{ p}
\bmdefine\Bq{ q}
\bmdefine\Br{ r}
\bmdefine\Bs{ s}
\bmdefine\Bt{ t}
\bmdefine\Bu{ u}
\bmdefine\Bv{ v}
\bmdefine\Bw{ w}
\bmdefine\Bx{ x}
\bmdefine\By{ y}
\bmdefine\Bz{ z}
\bmdefine\BA{ A}
\bmdefine\BB{ B}
\bmdefine\BC{ C}
\bmdefine\BD{ D}
\bmdefine\BE{ E}
\bmdefine\BF{ F}
\bmdefine\BG{ G}
\bmdefine\BH{ H}
\bmdefine\BI{ I}
\bmdefine\BJ{ J}
\bmdefine\BK{ K}
\bmdefine\BL{ L}
\bmdefine\BM{ M}
\bmdefine\BN{ N}
\bmdefine\BO{ O}
\bmdefine\BP{ P}
\bmdefine\BQ{ Q}
\bmdefine\BR{ R}
\bmdefine\BS{ S}
\bmdefine\BT{ T}
\bmdefine\BU{ U}
\bmdefine\BV{ V}
\bmdefine\BW{ W}
\bmdefine\BX{ X}
\bmdefine\BY{ Y}
\bmdefine\BZ{ Z}

\newcommand{\SFL}{\mathsf{L}}

\begin{document}
\title{A hint on the localization of the buckling deformation at vanishing curvature points on thin elliptic shells}
\author{Davit Harutyunyan\thanks{University of California Santa Barbara, harutyunyan@math.ucsb.edu}}
\maketitle

\begin{abstract}
The general theory of slender structure buckling by Grabovsky and Truskinovsky [\textit{Cont. Mech. Thermodyn.,} 19(3-4):211-243, 2007], (later extended in [\textit{Journal of Nonlinear Science.,} Vol. 26, Iss. 1, pp. 83--119, 2016] by Grabovsky and the author), predicts that the critical buckling load of a thin shell under dead loading is closely related to the Korn's constant (in Korn's first inequality) of the shell under the Dirichlet boundary conditions resulting from the loading program. It is known that under zero Dirichlet boundary conditions on the thin part of the boundary of positive, negative, and zero (one principal curvature vanishing, and one apart from zero) Gaussian curvature shells, the optimal Korn constant in Korn's first inequality scales like the thickness to the power of $-1, -4/3,$ and $-3/2$ respectively. In this work we analyse the scaling of the optimal constant in Korn's first inequality for elliptic shells that contain a finite number of points where both principal curvatures vanish. We prove that the presence of at least one such point on the shell leads to the scaling drop from the thickness to the power of $-1$ to the thickness to the power of $-3/2.$ To our best knowledge, this is the first result in the direction for constant-sign curvature shells, that do not contain a developable region. In addition, under the assumption that a suitable trivial branch exists, we prove that in fact the buckling deformation of such shells under dead loading, should be localized at the vanishing curvature points, as the shell thickness $h$ goes to zero. 

\end{abstract}

\section{Introduction}
\label{sec:1}
Let $h>0$ be a small parameter, and let $S\subset\mathbb R^3$ be a bounded, regular, connected, and smooth enough surface with nonempty relative interior. A shell of thickness $h$ around $S$ is the $h/2$ neighborhood of $S$ in the normal direction:
$$\Omega^h=\{x+t\Bn(x) \ : \ x\in S,\ t\in [-h/2,h/2]\},$$
where $\Bn(x)$ is the unit normal to $S$ at the point $x\in S.$\footnote{The direction of the normal does not matter here.}
The surface $S$ is called the mid-surface of the shell $\Omega^h.$
It is known that the critical buckling load of a shell under compression is closely related to the optimal constant in Korn's first inequality
for the shell [\ref{bib:Gra.Tru.},\ref{bib:Gra.Har.1},\ref{bib:Gra.Har.2}]. In particular, the authors have rigorously proven in [\ref{bib:Gra.Har.1}], that the critical buckling load of a circular cylindrical shell under axial compression is given by the optimal constant in a Korn-like inequality for the shell, thus recovering the well-known Koiter formula [\ref{bib:Koiter}]. In [\ref{bib:Gra.Har.2}], the authors have proven a new critical buckling load formula for slender structures under dead loading, under some conditions on the domain that are mainly of geometric nature (see Section~\ref{sec:5} for a detailed discussion of the matter). Motivated mainly by these applications to buckling, the present work is concerned with the study of the optimal constant in Korn's first inequality for certain types of shells. We also discuss on how the
optimizers in the inequality should behave, which is important from the engineering point of view. Korn's inequalities and their nonlinear analogue, called the geometric rigidity estimate [\ref{bib:Fri.Jam.Mue.1}], play a central role in the elasticity theory. Korn's first and second inequalities read as follows: \textit{Let $n\geq 2,$ and let the domain $\Omega\subset\mathbb R^n$ be open, bounded, connected, and Lipschitz. Denote by $\mathrm{skew}(\mathbb R^n)$ the vector subspace of $L(\mathbb R^n,\mathbb R^n)$ that
consist of mappings with a skew-symmetric gradient. Then
\begin{itemize}
\item[](\textbf{Korn's second inequality}). There exists a constant $C_2=C_2(\Omega),$ depending only on the domain $\Omega,$ such that for any vector field
$\Bu\in H^1(\Omega),$ one has the inequality
\begin{equation}
\label{1.1}
\|\nabla\Bu\|_{L^2(\Omega)}^2\leq C_2(\|e(\Bu)\|_{L^2(\Omega)}^2+\|\Bu\|_{L^2(\Omega)}^2),
\end{equation}
where $e(\Bu)=\frac{1}{2}(\nabla\Bu+\nabla\Bu^T)$ is the symmetric part of the gradient.
\item[](\textbf{Korn's first inequality}). Let the subspace $V\subset H^1(\Omega)$ be such that $V\cap \mathrm{skew}(\mathbb R^n)=\{0\}.$ There exists a constant $C_1=C_1(\Omega,V),$ depending only on the domain $\Omega$ and the subspace $V,$ such that for any vector field
$\Bu\in V,$ one has the inequality
\begin{equation}
\label{1.2}
\|\nabla\Bu\|_{L^2(\Omega)}^2\leq C_1\|e(\Bu)\|_{L^2(\Omega)}^2.
\end{equation}
\end{itemize}
}
Initially these inequalities were introduced by Korn to prove coercivity of the Linear elastic energy [\ref{bib:Korn.1},\ref{bib:Korn.2}].
However they have found further important applications in the studies of deformation of bodies [\ref{bib:Friedrichs},\ref{bib:Ciarlet},\ref{bib:Kohn}], fracture of bodies [\ref{bib:Cha.Con.Fra.}], in dimension reduction and buckling analysis of thin structures  [\ref{bib:Tov.Smi.},\ref{bib:Fri.Jam.Mue.1},\ref{bib:Gra.Tru.},\ref{bib:Gra.Har.1},\ref{bib:Mueller.}].  Also, the nonlinear counterpart of the constant $C_1$ in (\ref{1.2}) (and sometimes the constant $C_1$ itself) is referred to as the rigidity of the shell $\Omega,$ of course in the case when the domain $\Omega$ is a shell.
Recent progress in the study of the asymptotics of the optimal constants $C_2$ and $C_1$ in (\ref{1.1}) and (\ref{1.2}), has revealed
the result $C_2\sim h^{-1}$ for practically all $C^2$ shells (and even thin domains), and  $C_1\sim h^{-1},$  $C_1\sim h^{-4/3},$ and
$C_1\sim h^{-3/2}$ for positive, negative, and zero Gaussian curvature shells (with one principal direction zero, the other one apart from zero)
with a clamped thin face [\ref{bib:Gra.Har.3},\ref{bib:Harutyunyan.1},\ref{bib:Harutyunyan.2}].
This completely solved the problem of determining the asymptotics of $C_2$ in (\ref{1.1}) for shells regardless of the curvature, while (\ref{1.2})
is not studied for shells with sign-changing curvature, or for instance elliptic or hyperbolic shells that contain a zero curvature point, or finitely or infinitely many of them. In this work we will focus on the case of elliptic shells that contain finitely many points of zero principal (thus Gaussian) curvatures. We will show that in that case the asymptotics of the constant $C_1$ changes from $h^{-1}$ to $h^{-3/2},$ see Theorem~\ref{th:3.1}. The result in Theorem~\ref{3.1} can be regarded as a hint for understanding the buckling deformation of shells under consideration,
see Section~\ref{sec:5} for a detailed discussion.

\section{Preliminaries}
\setcounter{equation}{0}
\label{sec:2}

We will assume throughout the paper that the mid-surface $S$ is connected, bounded, regular, has nonempty relative interior, and is
$C^3$ smooth up to the boundary. We will also assume that $S$ can be given by a single patch parametrization $\Br=\Br(\Gth,z)$ in the principle coordinates, denoted by $\Gth$ and $z$, with the domain given by
\begin{equation}
\label{2.1}
E=\{(\Gth,z) \ : \ \Gth\in (0,1), \  z\in (z_1(\Gth),z_2(\Gth))\},
\end{equation}
where the functions $z_1(\Gth)$ and $z_2(\Gth)$ determining the boundary of the domain satisfy the uniform estimates:
\begin{equation}
\label{2.2}
0\leq z_1(\Gth)<z_2(\Gth)\leq 1,\quad l \leq z_2(\Gth)-z_1(\Gth)\quad\text{for all}\quad \Gth\in [0,1],
\end{equation}
for some constant $l>0.$ It is a known fact in elementary differential geometry that near any non-umbilical point $x$ (when the two principal curvatures are different at $x$) of a $C^2$ surface, there exists a local patch parametrization by principal coordinates that are orthogonal, while at umbilical points every direction can be regarded as principal. Note that (\ref{2.2}) and the imposed regularity on $S$ yield that $S$ does not have infinitesimally sharp edges. Introducing the coordinate $t$ in the normal to $S$ direction $\Bn,$ we obtain the set of orthonormal curvilinear coordinates $(t,\Gth,z)$ on $S,$ with the entire shell given by
$$\Omega^h=\left\{x+t\Bn(x)\ : \ x\in S, \ t\in(-h/2,h/2)\right\}.$$
Denote next $A_{\Gth}=\left|\frac{\partial \Br}{\partial \Gth}\right|, A_{z}=\left|\frac{\partial \Br}{\partial z}\right|,$
the two nonzero components of the metric tensor of the mid-surface and the two principal curvatures by $\Gk_{\Gth}$ and $\Gk_{z}$.
We will be utilizing the notation $f_{,\alpha}$ for the partial derivative $\frac{\partial}{\partial\alpha}$ in some cases, in particular inside the gradient matrix, to simplify the notation. It is easy to see [\ref{bib:Harutyunyan.1}], that for any vector field $\Bu\in H^1(\Omega^h,\mathbb R^3),$ the gradient is given by 
\begin{equation}
\label{2.3}
\nabla\Bu(t,\Gth,z)=
\begin{bmatrix}
  u_{t,t} & \dfrac{u_{t,\Gth}-A_{\Gth}\Gk_{\Gth}u_{\Gth}}{A_{\Gth}(1+t\Gk_{\Gth})} &
\dfrac{u_{t,z}-A_{z}\Gk_{z}u_{z}}{A_{z}(1+t\Gk_{z})}\\[3ex]
u_{\Gth,t}  &
\dfrac{A_{z}u_{\Gth,\Gth}+A_{z}A_{\Gth}\Gk_{\Gth}u_{t}+A_{\Gth,z}u_{z}}{A_{z}A_{\Gth}(1+t\Gk_{\Gth})} &
\dfrac{A_{\Gth}u_{\Gth,z}-A_{z,\Gth}u_{z}}{A_{z}A_{\Gth}(1+t\Gk_{z})}\\[3ex]
u_{ z,t}  & \dfrac{A_{z}u_{z,\Gth}-A_{\Gth,z}u_{\Gth}}{A_{z}A_{\Gth}(1+t\Gk_{\Gth})} &
\dfrac{A_{\Gth}u_{z,z}+A_{z}A_{\Gth}\Gk_{z}u_{t}+A_{z,\Gth}u_{\Gth}}{A_{z}A_{\Gth}(1+t\Gk_{z})}
\end{bmatrix},
\end{equation}
for $(t,\Gth,z)\in E\times \left(-\frac{h}{2},\frac{h}{2}\right),$ where $\Bu=(u_t,u_\Gth,u_z)$ is the representation of $\Bu$ in the local curvilinear coordinates 
$(t,\Gth,z).$ Next we introduce the so-called shell mid-surface parameters.
Denote
\begin{align}
\label{2.4}
&A=\sup_{(\Gth,z)\in E}(A_\Gth+A_z),\quad a=\inf_{(\Gth,z)\in E}\min(A_\Gth,A_z),\\ \nonumber
&B=\sup_{(\Gth,z)\in E}(|\nabla A_\Gth|+|\nabla A_z|), \quad L=\sup_{\Gth\in(0,1)}(z_2(\Gth)-z_1(\Gth)).
\end{align}
We will assume the conditions hold on $S:$
$$0<a,A,B,L<\infty,$$
and 
$$|\kappa_\Gth|,|\kappa_z|,|\nabla \kappa_\Gth|,|\nabla \kappa_z|\leq K<\infty,$$
for some $K>0.$ The constants $a,A,B,l,L,$ and $K$ will be called shell mid-surface parameters, with more to come in the sequel.
Another main assumption we make on the surface $S$ is that it is elliptic, containing finitely many points with $\Gk_\Gth=\Gk_z=0.$ Namely, we assume that $\Gk_\Gth,\Gk_z>0$ on $\bar S-\{x_0,x_1,\dots,x_n\}$ for some points $x_0,x_1,\dots,x_n\in S,$ such that $\Gk_\Gth(x_i)=\Gk_z(x_i)=0,$ for $i=0,1,2,\dots,n,$ where $\bar S$ is the closure of $S.$ This in particular implies that each point $x_i$ is a local minimizer for both curvatures $\Gk_\Gth$ and $\Gk_z,$ thus we have $|\nabla\Gk_\Gth(x_i)|=|\nabla\Gk_z(x_i)|=0,$ for $i=0,1,\dots,n.$ We will assume quadratic growth of principal curvatures at each point $x_i,$ meaning that
\begin{equation}
\label{2.5}
\frac{1}{c_1}|x-x_i|^2\leq \Gk_\Gth(x),\Gk_z(x)\leq c_1|x-x_i|^2, \quad x\in S,\ i=0,1,\dots,n,
\end{equation}
for some constant $c_1>0.$ Let us emphasize that if instead of quadratic growth we assume growths with possibly different speeds at $x_i,$ then still an analogous to Theorem~3.1 result holds, where the only difference will be the asymptotics of the optimal constant in (\ref{3.3}), see Remark~\ref{rem:3.2}. Note that due to the fact that $\Gk_\Gth$ and $\Gk_z$ are strictly positive away from the points $x_i,$ then by continuity and compactness it will be sufficient to assume that (\ref{2.5}) is fulfilled in some small neighborhoods of the points $x_i.$ 
Also, (\ref{2.5}) is clearly equivalent to the fact that each of the Hessian matrices $D^2\Gk_\Gth$ and $D^2\Gk_z$ is strictly positive definite  at all points $x_i.$ And finally, we will assume that the curvatures $\Gk_\Gth$ and $\Gk_z$ satisfy the following comparability condition:
\begin{equation}
\label{2.6}
\quad \left|\nabla\left(\frac{\Gk_\Gth}{\Gk_z}\right)\right|\leq c_2,
\end{equation}
on $S$ for some constant $c_2>0.$ 


\begin{definition}
\label{def:2.1}
The constants $a,A,B,K,l,L,c_1,$ and $c_2$ are called the shell mid-surface parameters.
\end{definition}

Finally, let us mention that in what follows, all the norms $\|\cdot\|$ are $L^{2}$ norms, and the $L^2$ inner product of two functions
$f,g\colon\Omega^h\to\mathbb R$ will be given by
\[
(f,g)_{\Omega^h}=\int_{\Omega^h}A_zA_\Gth f(t,\Gth,z)g(t,\Gth,z)d\Gth dzdt,
\]
which gives rise to the norm $\|f\|_{L^2(\Omega^h)}$, that is equivalent to the standard $L^2(\Omega^h)$ norm in the Euclidean setting.

\section{Main results}
\label{sec:3}
\setcounter{equation}{0}

As already mentioned in Section~\ref{sec:1}, we will impose zero Dirichlet boundary conditions on the displacement $\Bu$ on the thin part of the
boundary of $\Omega^h.$ To that end we denote
\begin{equation}
\label{3.1}
\partial_{S}\Omega^h=\{x+t\Bn(x) \ : \ x\in\partial S,\ t\in (-h/2,h/2)\}.
\end{equation}
The following admissible set will then be a subspace of $H^1(\Omega^h):$
\begin{equation}
\label{3.2}
V^h=\{\Bu\in H^1(\Omega^h) \ : \ \Bu=0 \ \text{on} \ \partial_{S}\Omega^h \ \text{in the sense of traces}\}.
\end{equation}
Let us point out that in what follows any constants $C,C_i>0,$ and $\tilde h>0$ will depend only on the shell mid-surface parameters (see Definition~\ref{def:2.1}).
One of the main results of the paper is a sharp Korn's first inequality providing an Ansatz free lower bound for displacements $\Bu$ satisfying the boundary conditions (\ref{3.2}).
\begin{theorem}[Korn's first inequality]
\label{th:3.1}
Assume the mid-surface $S\subset\mathbb R^3$ is connected, bounded, regular, has nonempty relative interior, and is
$C^3$ smooth up to the boundary. Assume $S$ can be given by a single patch parametrization $\Br=\Br(\Gth,z)$ in the principle coordinates $\Gth$ and $z,$ such that the conditions (\ref{2.1}) and (\ref{2.2}) are satisfied. Assume further that the conditions (\ref{2.5}) and (\ref{2.6}) are fulfilled as well. Then there exist constants $C,C_0,\tilde h>0,$ depending only on the shell mid-surface parameters, such that if $L\leq C_0$ (the shell is short in the $z$ direction), one has the Korn first inequality
\begin{equation}
  \label{3.3}
\|\nabla\Bu\|^2\leq \frac{C}{h^{3/2}}\|e(\Bu)\|^2,
\end{equation}
for all $h\in(0,\tilde h)$ and all $\Bu\in V^h.$ Moreover, the estimate (\ref{3.3}) is asymptotically sharp as $h\to 0$ in the sense,
that for any $h\in (0,1),$ there exists a displacement $\Bu^h\in V^h$ that turns (\ref{3.3}) into equality.
\end{theorem}

The below asymptotically sharp Korn-Poincar\'e inequality can be viewed as another important contribution of the paper, which is a
major factor in the proof of Theorem~\ref{th:3.1}, thanks to a universal Korn interpolation inequality (this will be recalled later) proven in [\ref{bib:Harutyunyan.2}].

\begin{theorem}[Korn-Poincar\'e inequality]
\label{th:3.2}
Under the conditions of Theorem~\ref{th:3.1}, there exist constants $C,C_0,\tilde h>0,$ depending only on the shell mid-surface parameters, such that if $L\leq C_0$ (the shell is short in the $z$ direction), then the Korn-Poincar\'e inequalities hold:
\begin{equation}
  \label{3.4}
  \|u_\Gth\|^2\leq C\|e(\Bu)\|^2,\quad \|u_z\|^2\leq C\|e(\Bu)\|^2,
\end{equation}
for all $h\in(0,\tilde h)$ and all $\Bu\in V^h.$ Moreover, the estimates in (\ref{3.4}) are asymptotically sharp as $h\to 0$ in the sense,
that for any $h\in (0,1),$ there exists a displacement $\Bu^h\in V^h$ that realizes the asymptotics of the constants in all the inequalities in (\ref{3.4}).
\end{theorem}

\begin{remark}
\label{rem:3.2}
It is not difficult to see that in the case when instead of quadratic growth in (\ref{2.5}), one has mixed growth conditions  
\begin{equation}
\label{3.4.1}
\frac{1}{c_1}|x-x_i|^{p_i}\leq \Gk_\Gth(x),\Gk_z(x)\leq c_1|x-x_i|^{p_i}, \quad x\in S,\ i=0,1,\dots,n,
\end{equation}
where $p_i\geq 2,$ then while the estimates in (\ref{3.4}) continue to hold, the estimate (\ref{3.3}) has to be replaced by 
\begin{equation*}
\|\nabla\Bu\|^2\leq \frac{C}{h^{\frac{2p+2}{p+2}}}\|e(\Bu)\|^2,\quad\text{where}\quad p=\max_{0\leq i\leq n} {p_i}.
\end{equation*} 
This can be verified by going over the proof of Theorem~\ref{th:3.1} with the choice $\alpha=\frac{1}{p+2}$ in (\ref{4.20}). 
The Ansatz in (\ref{4.34}) should then be chosen to be localized in the $h^\alpha$ neighborhood of the point\footnote{In case there exist multiple such points, we choose just one.} $x_j$ for which $p_j=p.$ This will also lead to the changes $\Gl(h)\sim h^{\frac{2p+2}{p+2}}$ and $\Gl(h)\leq C h^{\frac{2p+2}{p+2}}$ in (\ref{5.12}) and (\ref{5.13}) respectively. 
\end{remark}

Some remarks are in order.
\begin{remark}
\label{rem:3.3}
The condition $L<C_0$ as well as the assumption that the mid-surface $S$ is given by a single patch parametrization in the principal coordinates is purely of technical character, and in principle may be removed by working in a different local coordinate chart, see [\ref{bib:Yao}].
\end{remark}

\begin{remark}
\label{rem:3.4}
Both Theorems 3.1 and 3.2 hold true for thin domains of thickness of order $h$ as well. A passage from shells to thin domains can be carried out
employing the localization idea of Kohn and Vogelius [\ref{bib:Koh.Vog.}], see [Lemma~5.2, \ref{bib:Harutyunyan.2}] for details. Here, a thin domain is roughly a shell with nonconstant thickness.
\end{remark}

\section{Proof of main results}
\label{sec:4}
\setcounter{equation}{0}

\begin{proof}[Proof of Theorems~3.1 and 3.2] 
\textbf{Ansatz free lower bounds.} First of all we note that Theorem~\ref{th:3.2} will yield a major simplification in proving Theorem~\ref{th:3.1} due to the results proven in [\ref{bib:Harutyunyan.2}]. Indeed, it has been proven in [\ref{bib:Harutyunyan.2}, Theorem~3.1], that under the conditions of Theorem~\ref{th:3.1}, one has the following Korn interpolation inequality
\begin{equation}
\label{4.0}
\|\nabla\Bu\|^2\leq C\left(\frac{\|u_t\|\cdot\|e(\Bu)\|}{h}+\|\Bu\|^2+\|e(\Bu)\|^2\right),
\end{equation}
for all vector fields $\Bu\in H^1(\Omega^h).$ Hence, provided the bounds (\ref{3.4}) hold, one needs to only estimate the out-of-plane component $u_t$ to get (\ref{3.3}) from (\ref{4.0}). Therefore we will first focus on the estimates in (\ref{3.4}) in the sequel. For the sake of simplicity we will be working with the so-called simplified gradient $\BF$ given by
\begin{equation}
  \label{4.1}
\BF=
\begin{bmatrix}
  u_{t,t} & \dfrac{u_{t,\Gth}-A_{\Gth}\Gk_{\Gth}u_{\Gth}}{A_{\Gth}} &
\dfrac{u_{t,z}-A_{z}\Gk_{z}u_{z}}{A_{z}}\\[3ex]
u_{\Gth,t}  &
\dfrac{A_{z}u_{\Gth,\Gth}+A_{z}A_{\Gth}\Gk_{\Gth}u_{t}+A_{\Gth,z}u_{z}}{A_{z}A_{\Gth}} &
\dfrac{A_{\Gth}u_{\Gth,z}-A_{z,\Gth}u_{z}}{A_{z}A_{\Gth}}\\[3ex]
u_{ z,t}  & \dfrac{A_{z}u_{z,\Gth}-A_{\Gth,z}u_{\Gth}}{A_{z}A_{\Gth}} &
\dfrac{A_{\Gth}u_{z,z}+A_{z}A_{\Gth}\Gk_{z}u_{t}+A_{z,\Gth}u_{\Gth}}{A_{z}A_{\Gth}}
\end{bmatrix},
\end{equation}
instead of the full gradient $\nabla\Bu.$ The simplified gradient is obtained by inserting $t=0$ in the denominators
of the second and third column entries in the gradient matrix in (\ref{2.3}). The strategy is to target analogous to
(\ref{3.4}) estimates
\begin{equation}
  \label{4.2}
\|u_\Gth\|\leq C\|\BF^{sym}\|,\quad \|u_z\|\leq C\|\BF^{sym}\|,
\end{equation}
then utilizing the obvious bounds
\begin{equation}
  \label{4.3}
\|\BF^{sym}-e(\Bu)\|\leq \|\BF-\nabla \Bu\|\leq Ch\|\nabla \Bu\|,
\end{equation}
recover (\ref{3.4}), and thus (\ref{3.3}) too. Here $\BF^{sym}=\frac{1}{2}(\BF+\BF^T)$ is the symmetric part of $\BF.$ Let us point out, that by approximation we will assume that the displacement $\Bu$ is $C^\infty.$
\begin{proof}[Proof of (\ref{4.2})] Following the strategy in [\ref{bib:Harutyunyan.1}], we will be freezing the variable $t$ and work with the lower right block of the matrix $\BF.$ Fixing any $t\in(-h/2,h/2),$ we will treat all the functions as functions depending only on the variables $(\Gth,z)$ for now. The introduction of a function $\varphi=\varphi(z)\in C^1(S),$ depending only on $z$ and yet to be chosen, will allow us to derive some kind of Carleman estimates for the components $u_\Gth$ and $u_z$ in terms of the operator $T(\Bu)=\BF^{sym}.$
We also introduce the auxiliary inner product of two functions $f,g\in L^2(E)$ as
\begin{equation}
  \label{4.4}
(f,g)_E=\int_E A_\Gth A_z fgd\Gth dz.
\end{equation}
Denote for simplicity $\rho=\frac{\Gk_z}{\Gk_\Gth}.$ We can calculate
$$(\BF^{sym}_{22}, \varphi \rho u_z)_E=\int_E\varphi A_z\rho u_zu_{\Gth,\Gth}+\int_E\varphi A_\Gth A_z\Gk_z u_zu_t+
\int_E\varphi A_{\Gth,z} \rho u_z^2,$$
and
$$(\BF^{sym}_{33}, \varphi u_z)_E=\int_E\varphi A_\Gth u_zu_{z,z}+\int_E\varphi A_\Gth A_z\Gk_z u_zu_t+
\int_E\varphi A_{z,\Gth}u_\Gth u_z,$$
thus we get subtracting and eliminating the $u_t$ component:
\begin{align}
\label{4.5}
(\BF^{sym}_{22}, \varphi \rho u_z)_E&-(\BF^{sym}_{33}, \varphi u_z)_E=\\ \nonumber
&=\int_E\varphi A_z\rho u_zu_{\Gth,\Gth}+\int_E\varphi A_{\Gth,z} \rho u_z^2
-\int_E\varphi A_\Gth u_zu_{z,z}-\int_E\varphi A_{z,\Gth}u_\Gth u_z.
\end{align}
Next we aim to obtain a quadratic form in $u_\Gth$ and $u_z$ on the right hand side of (\ref{4.5}).
Using the boundary conditions we have integrating by parts,
\begin{align}
\label{4.6}
&\int_E\varphi A_z\rho u_zu_{\Gth,\Gth}=-\int_E \frac{\partial}{\partial\Gth}(\varphi A_z\rho) u_{\Gth}u_z -
\int_E\varphi A_z\rho u_\Gth u_{z,\Gth},\\ \nonumber
&\int_E\varphi A_\Gth u_zu_{z,z}=-\frac{1}{2}\int_E\frac{\partial}{\partial z}(\varphi A_\Gth)u_z^2.
\end{align}
The summand $\int_E\varphi A_z\rho u_\Gth u_{z,\Gth}$ is still in an undesired form, thus we will obtain the same quantity using the other
entry $\BF_{23}^{sym}.$ We have
$$A_zu_{z,\Gth}=2A_\Gth A_z\BF_{23}^{sym}+A_{\Gth,z}u_\Gth+A_{z,\Gth}u_z-A_\Gth u_{\Gth,z},$$
thus we can calculate the integral $\int_E\varphi A_z\rho u_\Gth u_{z,\Gth}$ as follows:
\begin{equation}
\label{4.7}
\int_E \varphi A_z\rho u_{\Gth}u_{z,\Gth}=\int_E [2A_\Gth A_z\varphi \rho u_{\Gth}\BF_{23}^{sym}+A_{\Gth,z}\varphi\rho u_\Gth^2+A_{z,\Gth}\varphi\rho u_\Gth u_z-A_\Gth\varphi\rho u_\Gth u_{\Gth,z}].
\end{equation}
Upon getting rid of the partial derivative in the integrand of the last summand in (\ref{4.7}), we obtain
\begin{equation}
\label{4.8}
\int_E (\varphi A_z\rho)u_{\Gth}u_{z,\Gth}=\int_E [2A_\Gth A_z\varphi \rho\BF_{23}^{sym} u_{\Gth}+
(A_{\Gth,z}\varphi\rho+\frac{1}{2}\frac{\partial}{\partial z}(A_\Gth\varphi\rho))u_\Gth^2+A_{z,\Gth}\varphi\rho u_\Gth u_z].
\end{equation}
Finally, putting together the identities (\ref{4.5}), (\ref{4.6}), and (\ref{4.8}) we arrive at the key identity
\begin{align}
\label{4.9}
(\rho\BF^{sym}_{22}-\BF^{sym}_{33}, \varphi u_z)_E&+(2\rho\BF_{23}^{sym}, \varphi u_{\Gth})_E\\ \nonumber
&=\int_E\left[\varphi A_{\Gth,z} \rho+\frac{1}{2}\frac{\partial}{\partial z}(\varphi A_\Gth)\right]u_z^2
-\int_E\left[\frac{1}{2}\frac{\partial}{\partial z}(A_\Gth\varphi\rho)+A_{\Gth,z}\varphi\rho\right]u_\Gth^2\\ \nonumber
&-\int_E\left[\frac{\partial}{\partial\Gth}(\varphi A_z\rho)+(1+\rho)\varphi A_{z,\Gth}\right]u_\Gth u_z.
\end{align}
Now, if we choose $\varphi(z)=e^{\lambda z},$ then for $\lambda>0$ large enough, the quadratic form in $(u_\Gth,u_z)$ on the right hand side of
(\ref{4.9}) will have a large positive coefficient of $u_z^2$ and a large negative coefficient of $u_\Gth^2$ due to (\ref{2.4}) and the bounds that follow. 
 Hence we have for $\lambda$ large enough, the Carleman-like estimates
\begin{align}
\label{4.10}
\int_E e^{\lambda z}|u_\Gth|^2 &\leq C\left( \frac{1}{\lambda} \int_E e^{\lambda z}|\BF^{sym}|^2+\int_E e^{\lambda z}|u_z|^2 \right),\\ \nonumber
\int_E e^{\lambda z}|u_z|^2 &\leq C\left( \frac{1}{\lambda} \int_E e^{\lambda z}|\BF^{sym}|^2+\int_E e^{\lambda z}|u_\Gth|^2 \right),\\ \nonumber
& \text{for }\quad \lambda\geq \lambda_0,
\end{align}
where $\lambda_0>0$ depends only the shell mid-surface parameters. Next we estimate the $u_t$ component in terms of $u_\Gth,$ $u_z,$ and $\BF^{sym}.$ To that end, note that if we eliminate the $u_t$ component from the $22$ and $33$ entries of $\BF^{sym},$ then the
inner product of the remaining expressions can be compared to the inner product of the $23$ and $32$ terms of $\BF$ by an integration by parts.
Namely, we can calculate integrating by parts,
\begin{align}
\label{4.11}
&(\BF^{sym}_{22}-\Gk_\Gth u_t, \BF^{sym}_{33}-\Gk_z u_t)_E
=\int_E\left(u_{\Gth,\Gth}+\frac{A_{\Gth,z}}{A_\Gth}u_z\right)\left(u_{z,z}+\frac{A_{z,\Gth}}{A_z}u_\Gth\right)\\ \nonumber
&=\int_E u_{\Gth,z}u_{z,\Gth}+\int_E\frac{A_{\Gth,z}A_{z,\Gth}}{A_\Gth A_z}u_zu_\Gth
-\frac{1}{2}\int_E\frac{\partial}{\partial z}\left(\frac{A_{\Gth,z}}{A_\Gth}\right)u_z^2
-\frac{1}{2}\int_E\frac{\partial}{\partial\Gth}\left(\frac{A_{z,\Gth}}{A_z}\right)u_\Gth^2.
\end{align}
We have on the other hand again integrating by parts,
\begin{align}
\label{4.12}
 (\BF_{23}&, 2\BF_{23}^{sym}-\BF_{23})_E
=\int_E\left(u_{z,\Gth}-\frac{A_{\Gth,z}}{A_\Gth}u_\Gth\right)\left(u_{\Gth,z}-\frac{A_{z,\Gth}}{A_z}u_z\right)\\ \nonumber
&=\int_E u_{\Gth,z}u_{z,\Gth}+\int_E\frac{A_{\Gth,z}A_{z,\Gth}}{A_\Gth A_z}u_zu_\Gth
+\frac{1}{2}\int_E\frac{\partial}{\partial z}\left(\frac{A_{\Gth,z}}{A_\Gth}\right)u_\Gth^2
+\frac{1}{2}\int_E\frac{\partial}{\partial \Gth}\left(\frac{A_{z,\Gth}}{A_z}\right)u_z^2.
\end{align}
Hence we obtain from (\ref{4.11}) and (\ref{4.12}) the identity
\begin{align*}
(\BF^{sym}_{22}-\Gk_\Gth u_t, \BF^{sym}_{33}-\Gk_zu_t)_E-&2(\BF^{sym}_{23},\BF_{23})_E+(\BF_{23},\BF_{23})_E\\
&=-\frac{1}{2}\int_\Omega\left(\frac{\partial}{\partial z}\left(\frac{A_{\Gth,z}}{A_\Gth}\right)+\frac{\partial}{\partial \Gth}\left(\frac{A_{z,\Gth}}{A_z}\right)\right)(u_\Gth^2+u_z^2).\\
\end{align*}
The magic is, that the coefficient of $u_\Gth^2+u_z^2$ in the above integral on the right hand side appears in the Codazzi-Gauss identity [\ref{bib:Lee},\ref{bib:Tov.Smi.}] given below:
\begin{align*}
\dif{}{ z}\left(\frac{A_{\Gth, z}}{A_{ z}}\right)+
\dif{}{\Gth}\left(\frac{A_{ z,\Gth}}{A_{\Gth}}\right)=-A_{ z}A_{\Gth}\Gk_{z}\Gk_{\Gth}.
\end{align*}
Hence, the Cauchy-Schwarz inequality, and the last equality together with the Codazzi-Gauss identities yield the estimate
\begin{equation}
\label{4.13}
(\BF^{sym}_{22}-\Gk_\Gth u_t, \BF^{sym}_{33}-\Gk_z u_t)_E\leq \|\sqrt{\Gk_\Gth\Gk_z}u_\Gth\|_{L^2(E)}^2+\|\sqrt{\Gk_\Gth\Gk_z}u_z\|_{L^2(E)}^2
+C\|\BF^{sym}\|_{L^2(E)}^2,
\end{equation}
for some constant $C>0.$ Denote next the Gaussian curvature by $K_G=\Gk_\Gth\Gk_z.$ From the uniform bounds (\ref{2.5}) we have
$$\frac{1}{C}\sqrt{K_G}\leq \Gk_\Gth,\Gk_z\leq C\sqrt{K_G},\quad \text{on}\quad S,$$
thus the estimate (\ref{4.13}) implies, upon some applications of the Schwarz inequality, the key bound
\begin{equation}
\label{4.14}
\|\sqrt{K_G}u_t\|_{L^2(E)}^2\leq C(\|\sqrt{K_G}u_\Gth\|_{L^2(E)}^2+\|\sqrt{K_G}u_z\|_{L^2(E)}^2+\|\BF^{sym}\|_{L^2(E)}^2).
\end{equation}
Next we aim to obtain a reverse to (\ref{4.14}) inequality. Recall that $L=\sup_{\Gth\in (0,1)}|z_2(\Gth)-z_1(\Gth)|.$ Poincar\'e's inequality in the $z$ direction implies
\begin{equation}
\label{4.15}
\frac{1}{L^2}\|u_z\|_{L^2(E)}^2\leq \|u_{z,z}\|_{L^2(E)}^2.
\end{equation}
On the other hand we have by the triangle inequality from the $33$ component of the matrix $\BF^{sym},$ that
\begin{equation}
\label{4.16}
\|u_{z,z}\|_{L^2(E)}^2\leq C(\|\BF^{sym}\|_{L^2(E)}^2+\|\sqrt{K_G}u_t\|_{L^2(E)}^2+\|u_\Gth\|_{L^2(E)}^2),
\end{equation}
thus (\ref{4.15}) and (\ref{4.16}) yield another key bound:
\begin{equation}
\label{4.17}
\|u_z\|_{L^2(E)}^2\leq CL^2(\|\BF^{sym}\|_{L^2(E)}^2+\|\sqrt{K_G}u_t\|_{L^2(E)}^2+\|u_\Gth\|_{L^2(E)}^2),
\end{equation}
Choosing $\lambda=\lambda_0$ in (\ref{4.10}) and utilizing the bounds (\ref{4.14}) and (\ref{4.17}), we discover
\begin{equation}
\label{4.18}
\|u_z\|_{L^2(E)}^2\leq CL^2e^{\lambda_0L}(\|\BF^{sym}\|_{L^2(E)}^2+\|u_z\|_{L^2(E)}^2).
\end{equation}
It is clear that if $L^2e^{\lambda_0L}\leq \frac{1}{2C}$ (which is equivalent to $L\leq C_0$ for some $C_0>0$) in (\ref{4.18}), then one gets from (\ref{4.18}) the bound $\|u_z\|_{L^2(E)}^2\leq C\|\BF^{sym}\|_{L^2(E)}^2.$ This will in tern imply due to (\ref{4.10}) a similar bound for
the $u_\Gth$ component. Concluding, we obtain
\begin{equation}
\label{4.19}
\|u_\Gth\|_{L^2(E)}+\|u_z\|_{L^2(E)}\leq C\|\BF^{sym}\|_{L^2(E)}\quad\text{as long as}\quad L\leq C_0,
\end{equation}
for some $C_0>0$ depending only on the shell mid-surface parameters. This establishes the bounds in (\ref{4.2}) upon integration in $t\in (-h,2,h/2).$ 

\end{proof}

Next we will prove a similar Korn-Poincar\'e estimate for the $u_t$ component too. To that end we will utilize (\ref{4.14}) and (\ref{4.19}). However, due to the fact that the Gaussian curvature vanishes at the points $x_i,$ and is small near each $x_i,$ then alone (\ref{4.14}) and (\ref{4.19}) do not do the job. For the sake of simplicity, we will assume that $n=0,$ i.e., there is only one point on $S$ where the principal curvatures vanish. It is easy to verify that the same proof works for the case $n\geq 1$ too with no additional tools. Let the point $x_0\in S$ correspond to the domain point $(\Gth_0,z_0)\in E.$ We will cut the shell into two parts, one being a small neighborhood of $x_0$, another one being its complement. Fixing an exponent $0<\alpha<1/2$ yet to be defined, denote the closed discs centered at $(\Gth_0,z_0)$ in $\mathbb R^2$ and with radii $h^\alpha$ and $2h^\alpha$ by $D_1$ and $D_2$ respectively. It is clear that the images $S_1=\Br(D_1)\subset S_2=\Br(D_2)\subset S$ have in-plane size of order $h^\alpha.$ Denote next by $\Omega_1^h$ and $\Omega_2^h$ the shells with thickness $h$ and mid-surface $S_1$ and $S_2$ respectively. We have on one hand by (\ref{2.6}) that
\begin{equation}
\label{4.20}
C_1h^{2\alpha}\leq \sqrt{K_G}\leq C_2\quad \text{on}\quad E-D_1,
\end{equation}
thus the estimate (\ref{4.14}) will imply upon integrating in $t\in (-h/2,h/2),$
\begin{equation}
\label{4.21}
\|u_t\|_{L^2(\Omega^h-\Omega_1^h)}\leq \frac{C}{h^{2\alpha}}(\|u_\Gth\|_{L^2(\Omega^h)}+\|u_z\|_{L^2(\Omega^h)}+\|\BF^{sym}\|_{L^2(\Omega^h)}).
\end{equation}
Note also that (\ref{4.19}) and (\ref{4.21}) imply
\begin{equation}
\label{4.22}
\|u_t\|_{L^2(\Omega^h-\Omega_1^h)}\leq \frac{C}{h^{2\alpha}}\|\BF^{sym}\|_{L^2(\Omega^h)}.
\end{equation}
Consequently invoking (\ref{4.3}), we discover the estimate
\begin{equation}
\label{4.23}
\|u_t\|_{L^2(\Omega^h-\Omega_1^h)}\leq \frac{C}{h^{2\alpha}}(\|e(\Bu)\|_{L^2(\Omega^h)}+h\|\nabla \Bu\|_{L^2(\Omega^h)}).
\end{equation}
Next, on order to get an estimate on the full gradient $\nabla\Bu$ in the domain $\Omega^h-\Omega_1^h,$ we employ the universal Korn interpolation inequality (\ref{4.0}). Recall that (\ref{4.0}) is true for any bounded $C^2$ shells $\Omega^h$ and any vector fields $\Bu\in H^1(\Omega^h),$ e.g.,
 [Theorem~3.1, \ref{bib:Harutyunyan.2}]. Namely, we have applying (\ref{4.0}) in $\Omega^h-\Omega_1^h$ and taking into account the estimate (\ref{4.23}), the bound 
 \begin{equation}
\label{4.24}
\|\nabla\Bu\|_{L^2(\Omega^h-\Omega_1^h)}^2\leq C\left((\frac{1}{h^{2\alpha+1}}+\frac{1}{h^{4\alpha}})\|e(\Bu)\|_{L^2(\Omega^h)}^2+
\frac{\|\nabla \Bu\|_{L^2(\Omega^h)}\|e(\Bu)\|_{L^2(\Omega^h)}}{{h^{2\alpha}}}+h^{2-4\alpha}\|\nabla \Bu\|_{L^2(\Omega^h)}^2 \right).
\end{equation} 
Note that the constant $C$ in (\ref{4.24}) depends only on the mid-surface parameters of the shell $\Omega^h-\Omega_1^h$ as proven in [Theorem~3.1, \ref{bib:Harutyunyan.2}], thus given the definition of $\Omega^h-\Omega_1^h$, the constant $C$ in (\ref{4.24}) is uniform in $h.$ Now, in order to estimate the missing part $\|u_t\|_{L^2(D_1)},$ we employ Korn's first inequality without boundary conditions in $\Omega_2^h.$
It is known that Korn's first inequality without boundary conditions holds with a constant $C/h^2$ for shells of thickness $h,$ regardless of the curvature, [\ref{bib:Fri.Jam.Mue.1}]. Namely, the following statement holds: \textit{Let $\hat{\Omega}\subset \mathbb R^3$ be a bounded and connected $C^2$ shell with mid-surface $\hat{S}$ and thickness $h.$ Then there exists a constant $C=C(\hat{S}),$ depending only the mid-surface $\hat{S}$ parameters, such that for any vector field $\Bu\in H^1(\hat{\Omega}),$ there exists a constant skew-symmetric matrix $\BA\in \mathrm{skew}(\mathbb R^{3\times 3})$ ($\BA^T=-\BA$) such that }
 \begin{equation}
\label{4.25}
\|\nabla\Bu-\BA\|_{L^2(\hat{\Omega})}^2 \leq \frac{C}{h^2}\|e(\Bu)\|_{L^2(\hat{\Omega})}^2.
\end{equation}
As $\alpha<1/2,$ then we can rescale the shell $\Omega_2^h$ by a factor of $h^\alpha$ to obtain a new shell with thickness $h^{1-\alpha}$ and in-plane sizes of
order one. Thus there exists a constant skew-symmetric matrix $\BA_2\in \mathbb R^{3\times 3},$ such that
\begin{equation}
\label{4.26}
\|\nabla \Bu-\BA_2\|_{L^2(\Omega_2^h)}\leq \frac{C}{h^{1-\alpha}}\|e(\Bu)\|_{L^2(\Omega_2^h)}.
\end{equation}
Next, we employ (\ref{4.24}) and (\ref{4.26}) to get by the triangle inequality:
\begin{align}
\label{4.27}
\|\BA_2\|_{L^2(\Omega_2^h-\Omega_1^h)}&\leq C(\frac{1}{h^{\alpha+1/2}}+\frac{1}{h^{2\alpha}}+\frac{1}{h^{1-\alpha}})\|e(\Bu)\|_{L^2(\Omega^h)}\\ \nonumber
&+C\frac{\sqrt{\|\nabla \Bu\|_{L^2(\Omega^h)}\|e(\Bu)\|_{L^2(\Omega^h)}}}{{h^{\alpha}}}+Ch^{1-2\alpha}\|\nabla \Bu\|_{L^2(\Omega^h)}.
\end{align}
We now choose the exponent $\alpha$ to optimize (\ref{4.27}), i.e., $\alpha=1/4.$ Thus (\ref{4.27}) yields after a suitable Cauchy inequality the estimate:
\begin{equation}
\label{4.28}
\|\BA_2\|_{L^2(\Omega_2^h-\Omega_1^h)}\leq
C\left(\frac{1}{h^{3/4}}\|e(\Bu)\|_{L^2(\Omega^h)}+h^{1/2}\|\nabla \Bu\|_{L^2(\Omega^h)}\right).
\end{equation}
Consequently, as the matrix $\BA_2$ is constant, and the domains $\Omega_2^h-\Omega_1^h$ and $\Omega_1^h$ have comparable volumes, we have a similar 
to (\ref{4.28}) estimate in the domain $\Omega_1^h$ too:
\begin{equation}
\label{4.29}
\|\BA_2\|_{L^2(\Omega_1^h)}\leq C\left(\frac{1}{h^{3/4}}\|e(\Bu)\|_{L^2(\Omega^h)}+h^{1/2}\|\nabla \Bu\|_{L^2(\Omega^h)}\right).
\end{equation}
Putting together now (\ref{4.26}), (\ref{4.28}), and (\ref{4.29}) we get by the triangle inequality, that
\begin{equation}
\label{4.30}
\|\nabla \Bu\|_{L^2(\Omega_2^h)}\leq C\left(\frac{1}{h^{3/4}}\|e(\Bu)\|_{L^2(\Omega^h)}+h^{1/2}\|\nabla \Bu\|_{L^2(\Omega^h)}\right).
\end{equation}
Finally, by combining (\ref{4.24}) and (\ref{4.30}) we arrive at 
\begin{equation}
\label{4.31}
\|\nabla \Bu\|_{L^2(\Omega^h)}\leq C\left(\frac{1}{h^{3/4}}\|e(\Bu)\|_{L^2(\Omega^h)}+h^{1/2}\|\nabla \Bu\|_{L^2(\Omega^h)}\right),
\end{equation}
which itself implies for small enough $h$ the desired bound (\ref{3.3}):
\begin{equation}
\label{4.32}
\|\nabla \Bu\|_{L^2(\Omega^h)} \leq \frac{C}{h^{3/4}}\|e(\Bu)\|_{L^2(\Omega^h)}.
\end{equation}
Also, note that once (\ref{3.3}) is proven, it will imply together with (\ref{4.19}) and (\ref{4.3}) the estimates in (\ref{3.4}) as well.
This completes the Ansatz-free lower bound parts of Theorems 3.1 and 3.2.\\
\textbf{The Ansatz.} The Ansatz realizing the asymptotics in (\ref{3.3}) and (\ref{3.4}) is going to be a Kirchhoff-like Ansatz, localized
at the zero curvature point $x_0.$ The first thing to note is that the inequalities in (\ref{3.4}) suggest that we can remove all $u_\Gth$ and $u_z$ terms (not with partial derivatives) from $\BF$ and work with the even more simplified matrix $\BF_1,$ given by
\begin{equation}
  \label{4.33}
\BF_1=
\begin{bmatrix}
  u_{t,t} & \dfrac{u_{t,\Gth}}{A_{\Gth}} & \dfrac{u_{t,z}}{A_{z}}\\[3ex]
u_{\Gth,t}  &
\dfrac{u_{\Gth,\Gth}}{A_{\Gth}}+\Gk_{\Gth}u_{t} & \dfrac{u_{\Gth,z}}{A_{z}}\\[3ex]
u_{ z,t}  & \dfrac{u_{z,\Gth}}{A_{\Gth}} & \dfrac{u_{z,z}}{A_{z}}+A_{\Gth}\Gk_{z}u_{t}
\end{bmatrix},
\end{equation}
which almost looks like the Euclidean gradient. Also, the above analysis suggests that, in order to have equalities everywhere, the Ansatz should be localized in the ball centered at the point $(\Gth_0,z_0)$ with radius $Ch^{1/4},$ see Theorem~\ref{th:5.4} in Section~\ref{sec:5}. Aiming to make the first row
of the symmetric part $\BF_1^{sym}$ zero, we arrive at the Ansatz
\begin{equation}
\label{4.34}
\begin{cases}
u_t&=f\left(\frac{\Gth-\Gth_0}{h^{1/4}},\frac{z-z_0}{h^{1/4}}\right),\\[2ex]
u_\Gth&=-\frac{t}{h^{1/4} A_\Gth}\cdot f_{,\xi}\left(\frac{\Gth-\Gth_0}{h^{1/4}},\frac{z-z_0}{h^{1/4}}\right),\\[2ex]
u_z&=-\frac{t}{h^{1/4} A_z}\cdot f_{,\eta}\left(\frac{\Gth-\Gth_0}{h^{1/4}},\frac{z-z_0}{h^{1/4}}\right),\\
\end{cases}
\end{equation}
where the function $f(\xi,\eta)\in C_0^2[-1,1]^2$ is such that all the norms $\|f\|_{L^2[-1,1]^2},$ $\|\nabla f\|_{L^2[-1,1]^2},$
and $\|D^2 f\|_{L^2[-1,1]^2}$ are of order one. Recalling that one has $0\leq\Gk_\Gth,\Gk_z\leq h^{1/2}$ in the $h^{1/4}$ neighborhood of the point $(\Gth_0,z_0)$, a simple calculation reveals for the gradient components, the symmetric part, and the
displacement components, the asymptotics:
\begin{align}
\label{4.35}
&\|(\nabla \Bu)_{11}\|_{L^2(\Omega^h)}^2=0,\\ \nonumber
&\|(\nabla \Bu)_{22}\|_{L^2(\Omega^h)}^2,\ \ \|(\nabla \Bu)_{23}\|_{L^2(\Omega^h)}^2,\ \ \|(\nabla \Bu)_{32}\|_{L^2(\Omega^h)}^2,
\ \ \|(\nabla \Bu)_{33}\|_{L^2(\Omega^h)}^2,\sim h^{3}\\ \nonumber
&\|(\nabla \Bu)_{12}\|_{L^2(\Omega^h)}^2,\ \ \|(\nabla \Bu)_{13}\|_{L^2(\Omega^h)},\ \ 
\|(\nabla \Bu)_{21}\|_{L^2(\Omega^h)}^2,\ \ \|(\nabla \Bu)_{31}\|_{L^2(\Omega^h)}^2\sim h^{3/2}\\ \nonumber
&\|\nabla \Bu\|_{L^2(\Omega^h)}^2\sim h^{3/2},\ \ \|e(\Bu)\|_{L^2(\Omega^h)}^2\sim h^3,\\ \nonumber
&\|u_t\|_{L^2(\Omega^h)}^2\sim h^2,\ \ \|u_\Gth\|_{L^2(\Omega^h)}^2\sim h^3,\ \  \|u_z\|_{L^2(\Omega^h)}^2\sim h^3.
\end{align}
Therefore we establish the upper bound parts in Theorems 3.1 and 3.2.

\end{proof}

\section{Remarks on the localization of the deformation}
\label{sec:5}
\setcounter{equation}{0}

In this section we discuss on why it should be favorable for the buckling deformation to localize at the points $x_i,$
where the principal curvatures vanish. As it was mentioned in Section~\ref{sec:1}, the critical buckling load of a slender structure, 
in particular shell, under dead loading is closely related to the constant $C_1$ in the inequality (\ref{1.1}). 
More precisely, assume the slender body\footnote{Note, that a body is called slender if $\lim_{h\to 0}C_1=\infty.$} $\Omega^h$ is subject to the dead
loading 
\begin{equation}
\label{5.1}
\Bt(x,\lambda)=\lambda\Bt(x)+O(\lambda^2)
\end{equation}
 at the boundary of $\Omega^h,$ where $\Bt(x)\colon\partial\Omega^h\to \mathbb R^3$ is the 
direction of the load $\Bt(x,\lambda),$ and the small parameter $\lambda>0$ is the magnitude of the load. 
The general theory of slender structure buckling by Grabovsky and Truskinovksy [\ref{bib:Gra.Tru.}] is aimed at determining the critical buckling load $\lambda(h)$ of $\Omega^h$ under the loading program (\ref{5.1}) as the loading parameter continuously increases from zero. In the theory of hyperelasticity, the total energy of the system
is given by 
\begin{equation}
\label{5.2}
E(\By)=\int_{\Omega^h}W(\nabla\By(x))dx-\int_{\partial\Omega^h}\By(s)\cdot \Bt(s)ds,
\end{equation}
where $W(F)\colon\mathbb R^3\to\mathbb R^{+}$ is the elastic energy density of the body $\Omega^h,$ satisfying the standard conditions imposed 
in the elasticity theory: 
\begin{itemize}
\item[(i)] No deformation no elastic energy: $W(\BI)=0$;
\item[(ii)] Absence of prestress: $W_{\BF}(\BI)=\Bzr$;
\item[(iii)] Frame indifference: $W(\BR\BF)=W(\BF)$ for every $\BR\in SO(3)$;
\item[(iv)] Local stability of the trivial deformation $\By(\Bx)=\Bx$:
  \begin{equation}
    \label{5.2.1}
    \av{\BL_{0}\BGx,\BGx}\geq \alpha|\BGx|^{2},\quad\BGx\in\mathrm{sym}(\bb{R}^{3})\quad \text{for some}\quad\alpha>0,
  \end{equation}
\end{itemize} 
where $\mathrm{sym}(\bb{R}^{3})$ is the space of $3\times 3$ symmetric matrices, and
$\BL_{0}=W_{\BF\BF}(\BI)$ is the linearly elastic tensor of the material. Note that condition (ii) is a direct consequence of (i) as the identity matrix is a global minimizer of the nonnegative density function $W.$ However as its mechanical meaning is absence of prestress, one typically list it as a separate condition.\\ 
The loading program (\ref{5.1}) results in a family of equilibrium deformations $\By(x,h,\lambda)$ of $\Omega^h,$ called a trivial branch, that minimize the total energy (\ref{5.2}) locally. The buckling then corresponds to the failure of the stability of the trivial branch $\By(x,h,\lambda),$  i.e., when the second variation of the energy becomes unstable at $\By(x,h,\lambda)$ for some value $\lambda=\lambda(h).$ The value $\lambda>0$ for which this happens for the first time is called the buckling load. The critical value $\lambda(h)$ for slender structures is typically given by a power low in $h,$ i.e., $\lambda(h)=Ch^{\alpha}$ for some $C,\alpha>0.$ The task is determining the exponent $\alpha,$ and if possible the coefficient $C,$ which is
a much more delicate questions and is addressed only for very simple geometries. The general theory in [\ref{bib:Gra.Tru.}] is applicable under some conditions on the family $\By(x,h,\lambda),$ that toughly speaking exclude large bending of the structure [\ref{bib:Gra.Tru.},\ref{bib:Gra.Har.1},\ref{bib:Gra.Har.2}]. Namely:
\begin{definition}
  \label{def:5.1}
  A family of Lipschitz equilibria $\By(x;h,\Gl)$ of $E(\By)$ is called a
  \textbf{linearly elastic trivial branch,} if there exist $h_{0}>0$ and $\Gl_{0}>0$,
  so that for every $h\in[0,h_{0}]$ and $\Gl\in[0,\Gl_{0}]$ one has:
\begin{itemize}
\item[(i)] $\By(x;h,0)=x.$
\item[(ii)] There exist a family of Lipschitz fields $\Bu^{h}(x)\colon\Omega^h\to\mathbb R^3$,
  independent of $\Gl$, such that
\begin{align}
  \label{5.3}
  &\|\Grad\By(x;h,\Gl)-\BI-\Gl\Grad\Bu^{h}(x)\|_{L^{\infty}(\GO^{h})}\le C\Gl^{2},\\ \nonumber
  &\left\|\dif{(\Grad\By)}{\Gl}(x;h,\Gl)-\Grad\Bu^{h}(x) \right\|_{L^{\infty}(\GO^h)}\le C\Gl,
  \end{align}
\end{itemize}
where the constant $C$ is independent of $h$ and $\Gl$.
\end{definition}
Extending the theory of Grabovsky and Truskinovsky, Grabovsky and the author came up with an explicit buckling load formula in [\ref{bib:Gra.Har.1}], under the same condition (\ref{5.3}) on the trivial branch $\By(x,h,\lambda),$ and an additional condition on the body, that is of geometric nature. Before formulating that buckling load formula, let us recall the necessary setup from [\ref{bib:Gra.Har.1},\ref{bib:Gra.Har.2}]. Let $\BGs_{h}$ be the linear elastic stress tensor $\BGs_{h}(x)=\BL_0e(\Bu^{h}(x)).$ One defines the set
\begin{equation}
  \label{5.4}
\CA^h=\left\{\BGf\in V^h:\av{\BGs_{h},\Grad\BGf^{T}\Grad\BGf}<0\right\}
\end{equation}
of potentially destabilizing variations\footnote{Variations that may make the second variation negative.}. Recall that the vector space $V^h,$ identifying the boundary conditions fulfilled by the displacement is defined in (\ref{3.2}). The \textbf{constitutively linearized
critical buckling load} $\Gl_{\rm cl}(h)$ is determined by minimizing the Rayleigh quotient
\begin{equation}
  \label{5.5}
\Gl_{\rm cl}(h)=-\inf_{\BGf\in\CA^h}\frac{\int_{\GO^h}\av{\SFL_{0}e(\BGf),e(\BGf)}d\Bx}
{\int_{\GO^h}\av{\BGs_{h},\Grad\BGf^{T}\Grad\BGf}d\Bx}.
\end{equation}
As already mentioned, there is a definition of \textbf{slenderness} of bodies, which reads as follows. 
\begin{definition}
\label{def:5.2}
We say that the  body $\GO_{h}$ is \textbf{slender} if
\begin{equation}
  \label{5.6}
  \lim_{h\to 0}K(V^h)=0,
\end{equation}
where $K(V^h)$ is the inverse of the "best" constant in (\ref{1.2}), i.e., 
$$K(V^h)=\inf_{\BGf\in V^h}\frac{\|e(\BGf)\|^2}{\|\nabla \BGf\|^2}.$$
\end{definition}
Let us recall, without getting too much into details, that the asymptotic (as $h\to 0$) minimizers $\BGf_{h}$ of (\ref{5.5}) are proven in [\ref{bib:Gra.Har.1}] to also be buckling modes for the loading program (\ref{5.1}). The following theorem [\ref{bib:Gra.Har.1}, Theorem~2.6] also provides an asymptotic formula for the critical buckling load $\lambda(h).$
\begin{theorem}
\label{th:5.3}
Suppose that the body is slender in the sense of Definition~\ref{def:5.2}. Assume that the constitutively linearized
critical buckling load $\Gl_{\rm cl}(h)$ satisfies $\Gl_{\rm cl}(h)>0$ for all sufficiently small $h$ and
\begin{equation}
  \label{5.7}
  \lim_{h\to 0}\frac{\Gl_{\rm cl}^2(h)}{K(V^h)}=0.
\end{equation}
Then $\Gl_{\rm cl}(h)$ is the buckling load, and any asymptotic (as $h\to 0$) minimizer $\BGf_{h}$ of (\ref{5.5}) is a buckling mode for the 
loading program (\ref{5.1}).
\end{theorem}
We emphasize that, we believe the condition (\ref{5.7}) should be satisfied as long as the existence of a trivial branch in Definition~\ref{def:5.1} is proven. Theorem~\ref{th:5.3} basically provides a mechanism of determining the buckling load and buckling modes, which are the asymptotic minimizers of (\ref{5.5}). This theory has been successfully applied in [\ref{bib:Gra.Tru.}] to determine the buckling load for a compressed thin column, and in [\ref{bib:Gra.Har.2}] to determine the buckling load for axially compressed circular cylindrical shells. Coming back to the case when the shell is elliptic with finitely many points with zero principal curvatures (as described in Section~\ref{sec:2}), we mention that some of the arguments to follow are heuristic, and we will not present a rigorous mathematical proof for them; those will be explicitly pointed out.\footnote{We believe the assumptions below without a proof should be possible to justify rigorously at least for vase-like shells $\Omega^h$ looking like cut spheres, etc.} Let now the shell $\Omega^h$ undergo dead loading as in (\ref{5.1}).

Assume further that the resulting deformations $\By(x,h,\lambda)$ satisfy the conditions in Definition~\ref{def:5.1}, so that the theory is applicable (note that we do not prove here the existence of such a family of equilibrium deformations $\By(x,h,\Gl)$). Due to the Lipschitz property of $\Bu^h$ in Definition~\ref{def:5.1}, the linear elastic stress tensor $\BGs_{h}(x)=\BL_0e(\Bu^{h}(x))$ is bounded in $L^\infty(\Omega^h)$ uniformly in $(h,\lambda)\in [0,\tilde h]\times[0,\lambda_0].$  
Thus from (\ref{5.5}) we have the lower bound
\begin{equation}
\label{5.9}
C_1\cdot\inf_{\BGf\in\CA^h}\frac{\|e(\BGf)\|^2}{\|\nabla \BGf\|^2}\leq \Gl_{\rm cl}(h).
\end{equation}
Consequently Korn's first inequality (\ref{3.3}) and (\ref{5.9}) yield the lower bound 
\begin{equation}
\label{5.10}
C_2h^{3/2}\leq \Gl_{\rm cl}(h),\quad C_2>0.
\end{equation}
Another assumption is that the Ansatz given in (\ref{4.34}) belongs to the admissible set $\CA^h$. Note that this will be
straightforward to check upon determination of the fields $\Bu^h$ in (\ref{5.3}). Hence, taking into account the scalings in (\ref{4.35}), 
we have from (\ref{5.5}) the upper bound 
\begin{equation}
\label{5.11}
\Gl_{\rm cl}(h)\leq C_3h^{3/2},\quad C_3>0.
\end{equation}
Note next that we have by (\ref{5.11}) and by Theorem~\ref{th:3.1}, that
$$\lim_{h\to 0}\frac{\Gl_{\rm cl}^2(h)}{K(V^h)}\leq\lim_{h\to 0}Ch^{3/2}=0,$$
thus Theorem~\ref{th:5.3} is applicable. Consequently, combining (\ref{5.10}) and (\ref{5.11}), we arrive at the formula for the buckling load $\Gl(h):$
\begin{equation}
\label{5.12}
\Gl(h)\sim h^{3/2}.
\end{equation}
Next we prove, that under the assumptions made (that the trivial branch $\By(x,h,\Gl)$ exists and satisfies the Definition~\ref{def:5.1}, 
and that the Ansatz in (\ref{4.34}) belongs to the admissible set $\CA^h$), the buckling modes, which are the asymptotic solutions of 
(\ref{5.5}), must be asymptotically (as $h\to 0$) localized in the $Ch^{1/4}$ neighborhoods of the zero curvature points $x_i,$ $i=0,1,\dots,n$ in a sense defined below. Assume again for simplicity that there is only one such point, i.e., $n=0.$ We will prove the following theorem.
\begin{theorem}
\label{th:5.4} 
Let the family of slender domain and loading pairs $(\Omega^h,\Bt^h),$ and resulting buckling modes $\Bu^h=(u_t^h,u_\Gth^h,u_z^h)\in V^h$ be such that 
\begin{equation}
\label{5.13}
\Gl(h)\leq C h^{3/2}\quad\text{for all}\quad h\in (0,h_0),
\end{equation}
for some constants $C,h_0>0.$ Denote by $(\Gth_0,z_0)\in E$ the pre-image of $x_0,$ and $S_\delta=\Br(D_\delta(\Gth_0,z_0))$ 
the image of the closed disc $D_\delta(\Gth_0,z_0)$ centered at the point $(\Gth_0,z_0)$ and with radius $\delta>0.$  
Let $\Omega^h(\delta)$ be the shell of thickness $h$ over the mid-surface $S_\delta.$ Then, there exists a constant $c>0,$ such that 
\begin{equation}
\label{5.14}
\lim_{h\to 0}\frac{\|\Bu^h\|_{L^2\left(\Omega^h(ch^{1/4})\right)}}{\|\Bu^h\|_{L^2(\Omega^h)}}=1.
\end{equation}
\end{theorem}

\begin{proof}
First of all note that the condition (\ref{5.13}) together with (\ref{5.9}) implies that the displacement family $\Bu^h$ realizes the asymptotics in the optimal constant in Korn's first inequality in (\ref{3.3}). Keeping in mind the Korn interpolation inequality (\ref{4.0}), and the uniform bounds in Theorem~\ref{th:3.2}, we infer that the family 
$\Bu^h$ has to fulfill the condition
\begin{equation}
\label{5.16}
\frac{\|e(\Bu^h)\|_{L^2(\Omega^h)}}{\|u_t^h\|_{L^2(\Omega^h)}}\sim h^{1/2}.
\end{equation}
Assume now in contradiction that (\ref{5.14}) fails. Then there exists $\tau>0$ and a sequence $h_k\to 0+,$ such that 
\begin{equation}
\label{5.18}
\frac{\|\Bu^k\|_{L^2\left(\Omega^{h_k}(kh_k^{1/4})\right)}}{\|\Bu^k\|_{L^2(\Omega^{h_k})}}\leq 1-\tau\quad\text{for all}\quad k\in\mathbb N,
\end{equation}
where we set for simplicity $\Bu^k=\Bu^{h_k}.$ We have by (\ref{2.5}) that $\sqrt{K_G}\geq k^2h_k^{1/2}$ in $E-D_{kh_k^{1/4}}(\Gth_0,z_0),$ thus the estimates in (\ref{4.14}) and (\ref{4.19}) imply upon integrating in $t\in (-h/2,h/2):$
$$\|u_t^k\|_{L^2\left(\Omega^{h_k}-\Omega^{h_k}(kh_k^{1/4})\right)}\leq \frac{C}{k^2h_k^{1/2}}\|\BF_k^{sym}\|_{L^2(\Omega^{h_k})},$$
which gives thanks to (\ref{4.3}) the bound 
\begin{equation}
\label{5.19}
\|u_t^k\|_{L^2\left(\Omega^{h_k}-\Omega^{h_k}(kh_k^{1/4})\right)}\leq 
\frac{C}{k^2h_k^{1/2}}\left(\|e(\Bu^k)\|_{L^2(\Omega^{h_k})}+h_k\|\nabla \Bu^k\|_{L^2(\Omega^{h_k})}\right).
\end{equation}
It is easy to see that the bounds (\ref{5.18}) and (\ref{4.19}) imply the estimate 
\begin{equation}
\label{5.20}
\|u_t^k\|_{L^2\left(\Omega^{h_k}(kh_k^{1/4})\right)}\leq \frac{1-\tau}{\tau}\|u_t^k\|_{L^2\left(\Omega^{h_k}-\Omega^{h_k}(kh_k^{1/4})\right)}+C\|e(\Bu^k)\|_{L^2(\Omega^{h_k})},
\end{equation}
for some $C>0.$ Consequently we derive from (\ref{5.19}) and (\ref{5.20}) the bound 
\begin{equation}
\label{5.21}
\|u_t^k\|_{L^2(\Omega^{h_k})}\leq C\left((1+\frac{1}{k^2h_k^{1/2}})\|e(\Bu^k)\|_{L^2(\Omega^{h_k})}+h_k\|\nabla \Bu^k\|_{L^2(\Omega^{h_k})}\right).
\end{equation}
Recalling Korn's first inequality (\ref{3.3}), we have 
$$h_k\|\nabla \Bu^k\|_{L^2(\Omega^h)}\leq Ch_k^{1/4}\|e(\Bu^k)\|_{L^2(\Omega^h)},$$
thus (\ref{5.21}) simplifies to 
\begin{equation}
\label{5.22}
\|u_t^k\|_{L^2(\Omega^{h_k})}\leq C\left(1+\frac{1}{k^2h_k^{1/2}}\right)\|e(\Bu^k)\|_{L^2(\Omega^{h_k})}
\end{equation}
Finally, it remains to note that (\ref{5.22}) implies 
$$\liminf_{k\to\infty}\frac{\|e(\Bu^k)\|_{L^2(\Omega^{h_k})}}{h_k^{1/2}\|u_t^k\|_{L^2(\Omega^{h_k})}}=\infty,$$
which contradicts (\ref{5.16}). The proof is now complete. 

\end{proof}

\section*{Acknowledgements}
This material is supported by the National Science Foundation under Grants No. DMS-1814361.

\end{document}